\documentclass{amsart}
\usepackage{amsmath,amssymb}
\usepackage{bm}
\usepackage{graphicx}
\usepackage{float}
\usepackage{ascmac}
\usepackage{fancybox}
\usepackage{tabularx}
\usepackage{multicol}
\usepackage{enumerate}
\usepackage{slashed}

\newtheorem{dfn}{Definition}[section]
\newtheorem{thm}[dfn]{Theorem}
\newtheorem{prop}[dfn]{Proposition}
\newtheorem{lem}[dfn]{Lemma}
\newtheorem{cor}[dfn]{Corollary}

\numberwithin{equation}{section}

\title[Lipschitz stability for an elliptic inverse problem]{Lipschitz stability for an elliptic inverse problem with two measurements}

\author{Mourad Choulli}
\address{Universit\'{e} de Lorraine, 34 cours L\'{e}opold, 54052 Nancy cedex, France}
\email{mourad.choulli@univ-lorraine.fr}

\author{Hiroshi Takase}
\address{Institute of Mathematics for Industry, Kyushu University, 744 Motooka, Nishi-ku, Fukuoka 819-0395, Japan}
\email{htakase@imi.kyushu-u.ac.jp}

\date{\today}
\begin{document}
\begin{abstract}
We consider the problem of determining the unknown boundary values of a solution of an elliptic equation outside a bounded open set $B$ from the knowledge of the values of this solution on a boundary of an arbitrary Lipschitz bounded domain surrounding $B$. We obtain for this inverse problem Lipschitz stability for an admissible class of unknown boundary functions. Our analysis applies as well to an interior problem. We also give an extension to the parabolic case.
\end{abstract}

\maketitle

\section{Introduction}\label{section1}

Let $n\ge 2$ be an integer. Throughout this text, we use the Einstein summation convention for quantities with indices. If in any term appears twice, as both an upper and lower index, that term is assumed to be summed from $1$ to $n$.

Let $(g_{ij})\in W^{1,\infty}(\mathbb{R}^n;\mathbb{R}^{n\times n})$ be a symmetric matrix-valued function satisfying, for some  $\theta>0$ 
\[
g_{ij}(x)\xi^i\xi^j\ge \theta|\xi|^2\quad x,\xi \in\mathbb{R}^n.
\]
Note that $(g^{ij})$  the matrix inverse to $g$ is uniformly positive definite as well. Let $p\in L^\infty(\mathbb{R}^n)$ and recall that the Laplace-Beltrami operator associated to  the metric tensor $g=g_{ij}dx^i\otimes dx^j$ is given by
\[
\Delta_g u:=\frac{1}{\sqrt{|g|}}\partial_i\left(\sqrt{|g|}g^{ij}\partial_ju\right),
\]
where $|g|=\mbox{det} (g)$.

Let $B\subset\mathbb{R}^n$ be a $C^2$ bounded open set with boundary $\mathcal{S}:=\partial B$. Set $U:=\mathbb{R}^n\setminus\overline{B}$ and consider the exterior boundary value problem
\begin{equation}\label{exBVP}
\begin{cases}
Pu:=-\Delta_gu+pu=0\quad &\text{in}\; U,
\\
u=\mathfrak{a}\quad &\text{on}\; \mathcal{S},
\end{cases}
\end{equation}
where $\mathfrak{a}\in H^{3/2}(\mathcal{S})$. When $p\ge \eta$ for some $\eta>0$, reducing first \eqref{exBVP} to a boundary value problem with homogeneous boundary condition and then applying Lax-Milgram lemma to the bilinear form associated with $P$, we derive that \eqref{exBVP} has a unique solution $u=u(\mathfrak{a})\in H^1(U)$. After transforming again \eqref{exBVP} to a boundary value problem with homogeneous boundary condition, we apply \cite[Theorem 9.25]{Brezis2011} and Remark 24 just after \cite[Theorem 9.26]{Brezis2011} to obtain that $u(\mathfrak{a})\in H^2(U)$.

Let $\Omega\Supset B$ be a Lipschitz bounded domain such that $\Omega\setminus\overline{B}$ is connected with boundary $\Gamma:=\partial\Omega$. We discuss the inverse problem of determining the unknown function $\mathfrak{a}$ from the knowledge of the two measurements $(u(\mathfrak{a})_{|\Gamma},\partial_\nu u(\mathfrak{a})_{|\Gamma})\in H^{3/2}(\Gamma)\times H^{1/2}(\Gamma)$, where $\nu$ denotes the outer unit normal to $\Gamma$.

This type of inverse problem setup is common in satellite gravitational gradiometry, where the gravitational potential of the Earth's surface is determined from observations of the gravitational potential and gravity on the satellite orbit (e.g., Pereverzev-Schock \cite{Pereverzev1999} and Freeden-Nashed \cite[Chapter 7]{Freeden2023}). 

We also discuss a similar inverse problem for the interior boundary value problem with a $C^2$ bounded domain $B$
\begin{equation}\label{inBVP}
\begin{cases}Pv=0\quad &\text{in}\; B,
\\
v=\mathfrak{a}\quad &\text{on}\; \mathcal{S},
\end{cases}
\end{equation}
when the measurement is made on $\Gamma:=\partial \Omega$, with a Lipschitz bounded open set $\Omega\Subset B$ such that $B\setminus\overline{\Omega}$ is connected. Henceforth, if $\mathfrak{a}\in H^{3/2}(\mathcal{S})$, the unique solution of \eqref{inBVP} will denoted by $v(\mathfrak{a})$. As for \eqref{exBVP}, after transforming again \eqref{inBVP} to a boundary value problem with homogeneous boundary condition, we get $v(\mathfrak{a})\in H^2(B)$ by applying \cite[Theorem 9.25]{Brezis2011}  and Remark 24 just after \cite[Theorem 9.26]{Brezis2011}. 

We now define the admissible set of the unknown function $\mathfrak{a}$. Let $\alpha >0$ and $\beta >0$ be fixed and define 
\[
\mathcal{A}=\{\mathfrak{a}\in H^{3/2}(\mathcal{S});\; \|\mathfrak{a}\|_{L^2(\mathcal{S})}\ge \alpha\;  \mbox{and}\; \|\nabla_\tau \mathfrak{a}\|_{L^2(\mathcal{S})}\le \beta\},
\]
where $\nabla_\tau$ denotes the tangential gradient.

Our main objective is to prove the following two theorems.

\begin{thm}\label{exLipschitz}
Let $B\subset\mathbb{R}^n$ be a $C^2$ bounded open set, $\Omega\Supset B$ be a Lipschitz bounded domain such that $\Omega\setminus\overline{B}$ is connected and set $\zeta_0=(g,p,B,\Omega,\beta/\alpha)$. Then there exists $C=C(\zeta_0)>0$ such that for any $\mathfrak{a}\in \mathcal{A}$ and $u(\mathfrak{a})\in H^2(U)$ satisfying \eqref{exBVP} we have
\[
\|\mathfrak{a}\|_{H^1(\mathcal{S})}\le C\left(\|u(\mathfrak{a})\|_{H^1(\Gamma)}+\|\partial_\nu u(\mathfrak{a})\|_{L^2(\Gamma)}
\right),
\]
where  $\nu$ denotes the outer unit normal to $\Gamma$.
\end{thm}

\begin{thm}\label{inLipschitz}
Let $B\subset\mathbb{R}^n$ be a $C^2$ bounded domain, $\Omega\Subset B$ be a Lipschitz bounded open set such that $B\setminus\overline{\Omega}$ is connected and set $\zeta_0=(g,p,B,\Omega,\beta/\alpha)$. Then there exists $C=C(\zeta_0)>0$ such that for any $\mathfrak{a}\in\mathcal{A}$ and $v(\mathfrak{a})\in H^2(B)$ satisfying \eqref{inBVP} we have
\[
\|\mathfrak{a}\|_{H^1(\mathcal{S})}\le C\left(\|v(\mathfrak{a})\|_{H^1(\Gamma)}+\|\partial_\nu v(\mathfrak{a})\|_{L^2(\Gamma)}
\right),
\]
where $\nu$ denotes the outer unit normal to $\Gamma$.
\end{thm}

Theorem \ref{exLipschitz} and \ref{inLipschitz} give the conditional Lipschitz stability estimates with the unknown function $\mathfrak{a}$ constrained to $\mathcal{A}$.
To our knowledge, these results give for the first time stability inequalities for the inverse problems we consider.

Theorems \ref{exLipschitz} and \ref{inLipschitz} still hold if instead of $P$ we take $P$ plus an operator of first order provided that the corresponding boundary value problems admit a unique $H^2$ solution.

\section{Carleman type inequality}\label{section2}

In this section, we prove a Carleman type inequality with a second large parameter that can be applied to both exterior problem \eqref{exBVP} and interior problem \eqref{inBVP}. As it was noted by many authors, the role of the second large parameter is to ensure the so-called H\"ormander's pseudo-convexity condition. The proof is similar to Choulli \cite{Choulli2016} in the estimate of terms inside a domain, but the estimate of boundary terms is original to this paper. For the reader's convenience, we provide  the main steps of the proof of this Carleman type inequality.

Let $D\subset\mathbb{R}^n$ be a Lipschitz bounded domain with boundary $\Lambda:=\partial D$ and $\Upsilon$ be a nonempty subboundary of $\Lambda$ so that $\Pi:=\Lambda \setminus \Upsilon$ has nonempty interior. Assume that there exists $\phi\in C^2(\overline{D})$ satisfying
\begin{equation}\label{weight_function}
\begin{cases}
 \phi>0,\quad \mbox{in}\; D,
\\
\phi_{|\Upsilon}=0,
\\
\displaystyle \delta:=\min_{\overline{D}}|\nabla\phi|>0.
\end{cases}
\end{equation}
Let $\nu$ be the outer unit normal to $\Lambda$ and recall that the tangential gradient $\nabla_\tau $ is defined by $\nabla_\tau w:=\nabla w-(\partial_\nu w)\nu$. The surface element on $\Lambda$ will be denoted by $dS$.

For convenience, we recall the following usual notations
\begin{align*}
&\langle X,Y\rangle=g_{ij}X^iY^j,\quad X=X^i\frac{\partial}{\partial x_i},\; Y=Y^i\frac{\partial}{\partial  x_i},
\\
&\nabla_gw=g^{ij}\partial_i w\frac{\partial}{\partial x_j},\quad w\in H^1(D),
\\
&|\nabla_gw|_g^2=\langle\nabla_gw,\nabla_gw\rangle=g^{ij}\partial_i w\partial_j w, \quad w\in H^1(D),
\\
&\nu_g=(\nu_g)^i\frac{\partial}{\partial x_j},\quad (\nu_g)^i=\frac{g^{ij}\nu_j}{\sqrt{g^{k\ell}\nu_k\nu_\ell}}
\\
&\partial_{\nu_g}w=\langle\nu_g,\nabla_g w\rangle, \quad w\in H^1(D).
\end{align*}
Also, define the tangential gradient $\nabla_{\tau_g} w$ with respect to $g$ by
\[
\nabla_{\tau_g} w:=\nabla_g w-(\partial_{\nu_g} w)\nu_g.
\]
We find that $|\nabla_{\tau_g} w|_g^2=|\nabla_g w|_g^2-|\partial_{\nu_g} w|^2$ holds.

\begin{prop}\label{global_Carleman_estimate}

Let $\zeta_1=(g,p,D,\Upsilon,\phi,\delta)$, $\varphi:=e^{\gamma\phi}$ and $\sigma:=s\gamma\varphi$. There exist $\gamma_\ast=\gamma_\ast(\zeta_1)>0$, $s_\ast=s_\ast (\zeta_1)>0$ and $C=C(\zeta_1)>0$ such that for any $\gamma \ge \gamma_\ast$, $s\ge s_\ast$ and $u\in H^2(D)$ we have
\begin{align*}
&C\left(\int_D e^{2s\varphi}\sigma(\gamma|\nabla u|^2+\gamma\sigma^2|u|^2)dx+\int_\Upsilon e^{2s\varphi}\sigma(|\partial_{\nu_g}u|^2+\sigma^2|u|^2)dS\right)
\\
&\hskip 1.5cm \le \int_D e^{2s\varphi}|Pu|^2dx+\int_\Pi e^{2s\varphi}\sigma(|\nabla u|^2+\sigma^2|u|^2)dS
\\
&\hskip 6cm+\int_\Upsilon e^{2s\varphi}\sigma|\nabla_\tau u|^2dS.\end{align*}
\end{prop}

\begin{proof}
As usual, it suffices to show the inequality when $p=0$. Let $u\in H^2(D)$, $z:=e^{s\varphi}u$ and $P_sz:=e^{s\varphi}L(e^{-s\varphi}z)$, where we set $L=\Delta_g$. By $L\varphi=\gamma\varphi(L\phi+\gamma|\nabla_g\phi|_g^2)$, a direct calculation yields $P_sz=P_s^+z+P_s^-z-\sigma L\phi z$, where
\begin{align*}
&P_s^+z:=Lz+s^2|\nabla_g\varphi|_g^2z,
\\
&P_s^-z:=-2s\langle\nabla_g\varphi,\nabla_g z\rangle-\gamma\sigma |\nabla_g\phi|_g^2 z.
\end{align*}

Let $dV_g= \sqrt{|g|}dx$ and endow $L^2(D)$ with following inner product 
\[
(v,w)_g:=\int_D uvdV_g.
\]
The norm associated to this inner product is denoted by $\|\cdot \|_g$.

Hereinafter, the integrals on $D$ are with respect to the measure $dV_g$ and those on $\Lambda$ are with respect to the surface measure $dS_g= \sqrt{|g|}dS$. Successive integrations by parts yield
\begin{align*}
&(P_s^+z,P_s^-z)_g
=\int_D2s\nabla_g^2\varphi(\nabla_g z,\nabla_g z)-\int_D \sigma L\phi|\nabla_g z|_g^2+\int_D\gamma\langle\nabla_g(\sigma|\nabla_g\phi|_g^2),\nabla_g z\rangle z
\\
&\hskip 2.5cm+\int_D\left[\sigma^3|\nabla_g\phi|_g^2L\phi +2s^3\nabla_g^2\varphi(\nabla_g\varphi,\nabla_g\varphi)\right]|z|^2
\\
&\hskip 3cm-\int_\Lambda \left[2s\langle\nabla_g\varphi,\nabla_g z\rangle+\gamma\sigma|\nabla_g\phi|_g^2 z\right]\partial_{\nu_g} z 
\\
&\hskip 3.5cm +\int_\Lambda s\partial_{\nu_g}\varphi|\nabla_g z|_g^2 -\int_\Lambda s^3|\nabla_g\varphi|_g^2\partial_{\nu_g}\varphi |z|^2
\end{align*}
and
\begin{align*}
&(P_s^+z,-\gamma\sigma |\nabla_g\phi|_g^2 z)_g
=\int_D\gamma\sigma |\nabla_g\phi|_g^2 |\nabla_g z|_g^2
\\
&\hskip 4cm +\int_D\gamma\langle\nabla_g(\sigma|\nabla_g\phi|_g^2),\nabla_g z\rangle z -\int_D\gamma\sigma^3 |\nabla_g\phi|_g^4 |z|^2
\\
&\hskip 5.5cm -\int_\Lambda \gamma\sigma\partial_{\nu_g} z|\nabla_g\phi|_g^2 z.
\end{align*}
Adding the above equalities yields
\begin{align*}
&(P_s^+z,P_s^-z)_g+(P_s^+z,-\gamma\sigma|\nabla_g\phi|_g^2 z)_g-\int_D 2\gamma\langle\nabla_g(\sigma|\nabla_g\phi|_g^2),\nabla_g z\rangle z
\\
&\hskip 1cm =\int_D\left[2s\nabla_g^2\varphi(\nabla_g z,\nabla_g z)+\sigma(-L\phi+\gamma|\nabla_g\phi|_g^2)|\nabla_g z|_g^2\right]
\\
&\hskip 1.5 cm +\int_D\left[\sigma^3|\nabla_g\phi|_g^2L\phi +2s^3\nabla_g^2\varphi(\nabla_g\varphi,\nabla_g\varphi)-\gamma\sigma^3 |\nabla_g\phi|_g^4\right]|z|^2
\\
&\hskip 2cm-\int_\Lambda 2\left[s\langle\nabla_g\varphi,\nabla_g z\rangle+\gamma\sigma|\nabla_g\phi|_g^2 z\right]\partial_{\nu_g} z
\\
&\hskip 2.5cm +\int_\Lambda s\partial_{\nu_g}\varphi|\nabla_g z|_g^2-\int_\Lambda s^3|\nabla_g\varphi|_g^2\partial_{\nu_g}\varphi |z|^2.
\end{align*}

Henceforth, $C=C(\zeta_1)>0$, $\gamma_\ast=\gamma_\ast(\zeta_1)>0$ and $s_\ast=s_\ast(\zeta_1)>0$ denote generic constants. 

By \eqref{weight_function}, we have
\[
|\nabla_g\phi|_g\ge C|\nabla\phi|\ge C\delta.
\]
Hence
\begin{align*}
&2s\nabla_g^2\varphi(\nabla_g z,\nabla_g z)+\sigma(-L\phi+\gamma|\nabla_g\phi|_g^2)|\nabla_g z|_g^2
\\
&\hskip 1.2cm =\sigma\left(2\nabla_g^2\phi(\nabla_g z,\nabla_g z)+2\gamma|\langle\nabla_g\phi,\nabla_g z\rangle|^2-L\phi|\nabla_g z|_g^2+\gamma|\nabla_g\phi|_g^2|\nabla_g z|_g^2\right)
\\
&\hskip 1.2cm \ge\sigma\left(2\nabla_g^2\phi(\nabla_g z,\nabla_g z)-L\phi|\nabla_g z|_g^2+\gamma|\nabla_g\phi|_g^2|\nabla_g z|_g^2\right)\\
&\hskip 1.2cm\ge C\gamma\sigma|\nabla_g z|_g^2,\quad \gamma \ge \gamma_\ast,
\end{align*}
and
\begin{align*}
&\sigma^3|\nabla_g\phi|_g^2L\phi +2s^3\nabla_g^2\varphi(\nabla_g\varphi,\nabla_g\varphi)-\gamma\sigma^3 |\nabla_g\phi|_g^4
\\
&\hskip 1.5cm=\sigma^3\left(2\nabla_g^2\phi(\nabla_g\phi,\nabla_g\phi)+\gamma|\nabla_g\phi|_g^4+L\phi|\nabla_g\phi|_g^2\right)
\\
&\hskip 1.5cm\ge C\gamma\sigma^3, \quad \gamma \ge \gamma_\ast.
\end{align*}
In consequence, we get
\begin{align*}&(P_s^+z,P_s^-z)_g+(P_s^+z,-\gamma\sigma|\nabla_g\phi|_g^2 z)_g-\int_D 2\gamma\langle\nabla_g(\sigma|\nabla_g\phi|_g^2),\nabla_g z\rangle z
\\
&\hskip 2cm\ge C\int_D\gamma\sigma|\nabla_g z|_g^2+C\int_D\gamma\sigma^3|z|^2-\mathcal{B},\quad \gamma \ge \gamma_\ast,
\end{align*}
where $\mathcal{B}:=\mathcal{B}_\Pi+\mathcal{B}_\Upsilon$, with
\begin{align*}\mathcal{B}_\Pi&:=\int_\Pi \left[2s\partial_{\nu_g} z\langle\nabla_g\varphi,\nabla_g z\rangle+2\gamma\sigma \partial_{\nu_g} z|\nabla_g\phi|_g^2 z \right.
\\
&\hskip 2cm \left.+s^3|\nabla_g\varphi|_g^2\partial_{\nu_g}\varphi|z|^2-s\partial_{\nu_g}\varphi|\nabla_g z|_g^2\right]
\end{align*}
and
\begin{align*}\mathcal{B}_\Upsilon&:=\int_\Upsilon \left[2s\partial_{\nu_g} z\langle\nabla_g\varphi,\nabla_g z\rangle+2\gamma\sigma \partial_{\nu_g} z |\nabla_g\phi|_g^2 z \right.
\\
&\hskip 2cm \left.+s^3|\nabla_g\varphi|_g^2\partial_{\nu_g}\varphi|z|^2-s\partial_{\nu_g}\varphi|\nabla_g z|_g^2\right].
\end{align*}

We check that
\[
\mathcal{B}_\Pi \le C\int_\Pi \sigma(|\nabla_g z|_g^2+\sigma^2|z|^2).
\]
On the other hand, according to \eqref{weight_function}, we have
\[
|\nabla_g z|_g^2=|\nabla_{\tau_g}z|_g^2+|\partial_{\nu_g} z|^2,\quad \langle\nabla_g\varphi,\nabla_g z\rangle=-\gamma\varphi|\nabla_g\phi|_g\partial_{\nu_g} z\quad \mbox{on}\; \Upsilon
\]
and $\nu_g=-\frac{\nabla_g\phi}{|\nabla_g\phi|_g}$ on $\Upsilon$. Whence
\begin{align*}
\mathcal{B}_\Upsilon&=\int_\Upsilon \left[-2\sigma|\nabla_g\phi|_g|\partial_{\nu_g} z|^2+2\gamma\sigma\partial_{\nu_g} z|\nabla_g\phi|_g^2 z \right.
\\
&\hskip 2.5cm \left. -\sigma^3|\nabla_g\phi|_g^3|z|^2+\sigma|\nabla_g\phi|_g|\nabla_g z|_g^2\right]
\\
&=\int_\Upsilon \sigma|\nabla_g\phi|_g\left[-|\partial_{\nu_g} z|^2+2\gamma\partial_{\nu_g} z|\nabla_g\phi|_g z \right.
\\
&\hskip 2.5cm\left. -\sigma^2|\nabla_g\phi|_g^2|z|^2+|\nabla_{\tau_g}z|_g^2\right],
\end{align*}
which means 
\begin{align*}
&\mathcal{B}_\Upsilon+(1/2)\int_\Upsilon \sigma|\nabla_g\phi|_g|\partial_{\nu_g}z|^2
\\
&\hskip 2cm =\int_\Upsilon\sigma|\nabla_g\phi|_g\left[-(1/2)|\partial_{\nu_g} z|^2+2\gamma\partial_{\nu_g} z|\nabla_g\phi|_g z \right.
\\
&\hskip 4.5cm\left. -\sigma^2|\nabla_g\phi|_g^2|z|^2+|\nabla_{\tau_g}z|_g^2\right].
\end{align*}
We note that
\begin{align*}
&-(1/2)|\partial_{\nu_g} z|^2+2\gamma\partial_{\nu_g} z|\nabla_g\phi|_g z
\\
&\hskip 2cm\le-(1/2)\left|\partial_{\nu_g} z-2\gamma|\nabla_g\phi|_g z\right|^2+C\gamma^2|z|^2
\\
&\hskip 2cm\le C\gamma^2|z|^2.
\end{align*}
Therefore, we obtain
\[\mathcal{B}_\Upsilon+(1/2)\int_\Upsilon\sigma|\nabla_g\phi|_g|\partial_{\nu_g}z|^2\le \int_\Upsilon\sigma|\nabla_g\phi|_g(C\gamma^2|z|^2-\sigma^2|\nabla_g\phi|_g^2|z|^2+|\nabla_{\tau_g}z|_g^2)
\]
and then
\[
\mathcal{B}_\Upsilon+C\int_\Upsilon\sigma(|\partial_{\nu_g}z|^2+\sigma^2|z|^2)\le \int_\Upsilon\sigma|\nabla_g\phi|_g|\nabla_{\tau_g}z|_g^2,\quad s\ge s_\ast.
\]
Combining the estimates above, we get 
\begin{align*}
&C\left( \int_D\sigma (\gamma|\nabla_g z|_g^2+\gamma\sigma^2|z|^2)+\int_\Upsilon\sigma(|\partial_{\nu_g}z|^2+\sigma^2|z|^2)\right)
\\
&\hskip 1cm \le  (P_s^+z,P_s^-z)_g+(P_s^+z,-\gamma\sigma|\nabla_g\phi|_g^2 z)_g
\\
&\hskip 3cm-\int_D 2\gamma\langle\nabla_g(\sigma|\nabla_g\phi|_g^2),\nabla_g z\rangle z+\mathcal{B}_\Pi+\int_\Upsilon\sigma|\nabla_g\phi|_g|\nabla_{\tau_g}z|_g^2 
\\
&\hskip 1cm \le (1/2)\|P_s^+ z+P_s^- z\|_g^2+C\int_D\gamma^2\varphi(|\nabla_g z|_g^2+\sigma^2|z|^2)
\\
&\hskip 3cm+C\int_{\Pi}\sigma(|\nabla_g z|_g^2+\sigma^2|z|^2)+C\int_\Upsilon\sigma|\nabla_{\tau_g}z|_g^2,\quad s\ge s_\ast.
\end{align*}
By
\[\|P_s^+z+P_s^- z\|_g^2=\|P_s z+\sigma L\phi z\|_g^2\le2\|P_s z\|_g^2+2\|\sigma L\phi z\|_g^2,\]
we have
\begin{align*}
&C\left( \int_D\sigma (\gamma|\nabla_g z|_g^2+\gamma\sigma^2|z|^2)+\int_\Upsilon\sigma(|\partial_{\nu_g}z|^2+\sigma^2|z|^2)\right)
\\
&\hskip 1cm \le \|P_s z\|_g^2+\int_D\gamma^2\varphi(|\nabla_g z|_g^2+\sigma^2|z|^2)
\\
&\hskip 4cm+\int_{\Pi}\sigma(|\nabla_g z|_g^2+\sigma^2|z|^2)+\int_\Upsilon\sigma|\nabla_{\tau_g}z|_g^2,\quad s\ge s_\ast.
\end{align*}
As the second term in the right-hand side  can absorbed by the left-hand side, we have in particular
\begin{align*}
&C\left( \int_D\sigma (\gamma|\nabla_g z|_g^2+\gamma\sigma^2|z|^2)+\int_\Upsilon\sigma(|\partial_{\nu_g}z|^2+\sigma^2|z|^2)\right)
\\
&\hskip 2cm\le \|P_s z\|_g^2+\int_\Pi \sigma(|\nabla_g z|_g^2+\sigma^2|z|^2)+\int_\Upsilon\sigma|\nabla_{\tau_g}z|_g^2,\quad s\ge s_\ast.
\end{align*}
Since $u=e^{-s\varphi}z$ and $\nabla_{\tau_g}u=e^{-s\varphi}\nabla_{\tau_g}z$ holds by $\nabla_{\tau_g}\phi=0$ on $\Upsilon$, we end up getting
\begin{align*}
&C\left( \int_D e^{2s\varphi}\sigma (\gamma|\nabla_g u|_g^2+\gamma\sigma^2|u|^2)+\int_\Upsilon e^{2s\varphi}\sigma(|\partial_{\nu_g}u|^2+\sigma^2|u|^2)\right)
\\
&\hskip 1cm\le \int_D e^{2s\varphi}|Lu|^2+\int_\Pi e^{2s\varphi}\sigma(|\nabla_g u|_g^2+\sigma^2|u|^2)+\int_\Upsilon e^{2s\varphi}\sigma|\nabla_{\tau_g}u|_g^2, \quad s\ge s_\ast.
\end{align*}
Equivalently, we have
\begin{align*}
&C\left( \int_D e^{2s\varphi}\sigma (\gamma|\nabla u|^2+\gamma\sigma^2|u|^2)dx+\int_\Upsilon e^{2s\varphi}\sigma(|\partial_{\nu_g}u|^2+\sigma^2|u|^2)dS\right)
\\
&\hskip 1cm\le \int_D e^{2s\varphi}|Lu|^2dx+\int_\Pi e^{2s\varphi}\sigma(|\nabla u|^2+\sigma^2|u|^2)dS
\\
&\hskip 6.5cm+\int_\Upsilon  e^{2s\varphi}\sigma|\nabla_\tau u|^2dS,\quad s\ge s_\ast.
\end{align*} 
The proof is then complete.
\end{proof}

\section{Proofs of main results}\label{section3}

Before proving Theorem \ref{exLipschitz} and Theorem \ref{inLipschitz}, we show that a weight function $\phi$ satisfying \eqref{weight_function} can be constructed for each problem in order to apply Proposition \ref{global_Carleman_estimate}.

\begin{lem}\label{exB}
Let $B$ and $\Omega$ be the open sets satisfying the same assumptions as in Theorem \ref{exLipschitz}. Then there exists $\phi\in C^2(\overline{\Omega\setminus B})$ satisfying
\[
\begin{cases}
\phi>0,\quad \mbox{in}\; \Omega\setminus\overline{B}
\\
 \phi_{|\mathcal{S}}=0,
\\
\displaystyle \delta:=\min_{ \overline{\Omega\setminus B}}|\nabla\phi| >0.
\end{cases}
\]
\end{lem}
\begin{proof}
Let $B_R\Supset\Omega$ be an open ball centered at $0$ with radius $R>0$. Applying  \cite[Theorem 9.4.3]{Tucsnak2009} for $\mathcal{O}:=B_R\setminus\overline{\Omega}\subset B_R\setminus\overline{B}$, we obtain the desired function $\phi$.
\end{proof}

\begin{lem}\label{inB}
Let $B$ and $\Omega$ be the open sets satisfying the same assumptions as in Theorem \ref{inLipschitz}. Then there exists $\phi\in C^2(\overline{B})$ satisfying
\[
\begin{cases}
\phi>0,\quad \mbox{in}\; B
\\
\phi_{|\mathcal{S}}=0,
\\
\displaystyle \delta:=\min_{ \overline{B\setminus\Omega}}|\nabla\phi|>0.
\end{cases}
\]
\end{lem}
\begin{proof}
As in the preceding lemma, we get $\phi$ with the required properties by applying \cite[Theorem 9.4.3]{Tucsnak2009} with $\mathcal{O}:=\Omega\subset B$.
\end{proof}

\begin{proof}[Proof of Theorem \ref{exLipschitz}]
Let $\mathfrak{a}\in \mathcal{A}$ and $u=u(\mathfrak{a})$. In this case, we have
\begin{equation}\label{t1.1}
\|\nabla_\tau u\|_{L^2(\mathcal{S})}=\|\nabla_\tau \mathfrak{a}\|_{L^2(\mathcal{S})}\le (\beta /\alpha)\|\mathfrak{a}\|_{L^2(\mathcal{S})}.
\end{equation}

By Lemma \ref{exB}, there exist $\phi\in C^2(\overline{\Omega\setminus B})$ such that \eqref{weight_function} is satisfied for $D:=\Omega\setminus\overline{B}$ and $\Upsilon:=\mathcal{S}$. Fix $\gamma>\gamma_\ast$, where $\gamma_\ast$ is given by Proposition \ref{global_Carleman_estimate}. Henceforth, $C=C(\zeta_0)>0$ denotes a generic constant. Using \eqref{t1.1}, $\varphi_{|\mathcal{S}}=1$ and applying Proposition \ref{global_Carleman_estimate} to $u$, we get
\begin{align*}
&Ce^{2s}s^3\|\mathfrak{a}\|_{L^2(\mathcal{S})}^2=C\int_{\mathcal{S}}e^{2s\varphi}s^3|\mathfrak{a}|^2dS
\\
&\hskip 1cm \le \int_\Gamma e^{2s\varphi}(s|\nabla u|^2+s^3|u|^2)dS+\int_{\mathcal{S}}e^{2s\varphi}s|\nabla_\tau \mathfrak{a}|^2dS
\\
&\hskip 1cm \le e^{Cs}(\|u\|_{H^1(\Gamma)}^2+\|\partial_\nu u\|_{L^2(\Gamma)}^2)+(\beta/\alpha)^2 e^{2s}s\|\mathfrak{a}\|_{L^2(\mathcal{S})}^2,\quad s\ge s_\ast,
\end{align*}
where $s_\ast=s_\ast(\zeta_0)>0$ is a constant.

Upon modifying $s_\ast$, we  may and do assume that $Cs^3-(\beta/\alpha)^2s>(C/2)s^3$. In this case we have
\[
e^{2s}s^3\|\mathfrak{a}\|_{L^2(\mathcal{S})}^2
 \le e^{Cs}(\|u\|_{H^1(\Gamma)}+\|\partial_\nu u\|_{L^2(\Gamma)})^2, \quad s\ge s_\ast,
\]
which implies
\[
\|\mathfrak{a}\|_{L^2(\mathcal{S})}\le e^{Cs_\ast}(\|u\|_{H^1(\Gamma)}+\|\partial_\nu u\|_{L^2(\Gamma)}).
\]
We complete the proof by using \eqref{t1.1} again.
\end{proof}

\begin{proof}[Proof of Theorem \ref{inLipschitz}]
By Lemma \ref{inB}, there exist $\phi\in C^2(\overline{B})$ such that \eqref{weight_function} is satisfied for $D:=B\setminus\overline{\Omega}$ and $\Upsilon:=\mathcal{S}$. The proof completes by applying Proposition \ref{global_Carleman_estimate} as well as the proof of Theorem \ref{exLipschitz}.
\end{proof}

\section{Extension to the parabolic case}\label{section4}

We limit ourselves to the case $B\Supset \Omega$, where $B$ is a bounded domain and $\Omega$ is a Lipschitz bounded open set such that $B\setminus\overline{\Omega}$ is connected. The notations of this section are the same as in the preceding ones. We further assume that $B$ is $C^\infty$ and $g\in C^\infty(\overline{B},\mathbb{R}^{n\times n})$. We need this regularity because we use the solvability of non-homogeneous initial-boundary value problems given in \cite{Lions1972b}. Clearly, this regularity of $B$ and the metric $g$ can be weakened.

Let $T>0$, $\mathcal{S}:=\partial B$, $Q:=(0,T)\times B$ and $\Sigma:=(0,T)\times \mathcal{S}$,  and  recall that

\begin{align*}
&H^{2,1}(Q)=L^2((0,T);H^2(B))\cap H^1((0,T);L^2(B)),
\\
&H^{3/2,3/4}(\Sigma)=L^2((0,T);H^{3/2}(\mathcal{S}))\cap H^{3/4}((0,T);L^2(\mathcal{S})).
\end{align*}

Let $\alpha >0$ and $u_0\in H^3(B)$ be arbitrarily fixed so that $\|u_0\|_{L^2(\mathcal{S})}\ge \alpha$. Consider the initial-boundary value problem
\begin{equation}\label{inIBVP}
\begin{cases}
(\partial_t-\Delta_g)u=0\quad &\text{in}\; Q,
\\
u=\mathfrak{g}\quad &\text{on}\; \Sigma,
\\
u(0, \cdot)=u_0.
\end{cases}
\end{equation}

For any $\mathfrak{g}\in H^{3/2,3/4}(\Sigma)$ satisfying $\mathfrak{g}(0,\cdot)_{|\mathcal{S}}=u_0{_{|\mathcal{S}}}$ (note that $\mathfrak{g}(0,\cdot)_{|\mathcal{S}}$ is an element of $H^{1/2}(\mathcal{S})$ by trace theorems in \cite{Lions1972b}) the initial-boundary value problem \eqref{inIBVP} has a unique solution $u=u(\mathfrak{g})\in H^{2,1}(Q)$ (e.g. \cite[Remark 4.1]{Lions1972b}). Further, if we assume in addition that $\partial_t\mathfrak{g}\in H^{3/2,3/4}(\Sigma)$ satisfies $\partial_t\mathfrak{g}(0,\cdot)_{|\mathcal{S}}=\Delta_g u_0{_{|\mathcal{S}}}$ then, using that $\partial_tu(\mathfrak{g})$ is the solution of the initial-boundary value problem \eqref{inIBVP} when $\mathfrak{g}$ and $u_0$ are replaced by $\partial_t\mathfrak{g}$ and $\Delta_g u_0$, we derive that $\partial_tu(\mathfrak{g})\in H^{2,1}(Q)$. Define
\begin{align*}
&\mathcal{G}_0=\{\mathfrak{g}\in H^{3/2,3/4}(\Sigma);\; \partial_t\mathfrak{g}\in H^{3/2,3/4}(\Sigma),
\\
&\hskip 5cm \mathfrak{g}(0,\cdot)_{|\mathcal{S}}=u_0{_{|\mathcal{S}}},\; \partial_t\mathfrak{g}(0,\cdot)_{|\mathcal{S}}=\Delta_g u_0{_{|\mathcal{S}}}\}.
\end{align*}

The inverse problem we consider in this section consists in determining the unknown function $\mathfrak{g}$ from the two measurements $(u(\mathfrak{g})_{|\Sigma_0}, \partial_{\nu}u(\mathfrak{g})_{|\Sigma_0})$, where $\Sigma_0:=(0,T)\times \Gamma$ and $\Gamma:=\partial\Omega$.

Before stating the main result of this section, we define the admissible set of the unknown coefficient $\mathfrak{g}$. For fixed $\beta>0$, we set
\[
\mathcal{G}:=\{\mathfrak{g}\in C^1(\overline{\Sigma})\cap \mathcal{G}_0;\; \|\mathfrak{g}(t,\cdot)\|_{L^2(\mathcal{S})}\ge \alpha,\; \|\partial_t\mathfrak{g}(t,\cdot)\|_{L^2(\mathcal{S})}+\|\nabla_\tau \mathfrak{g}(t,\cdot)\|_{L^2(\mathcal{S})}\le \beta\}.
\]

\begin{thm}\label{inHolder}
Let $\zeta_2=(g,T,B,\Omega,\beta/\alpha)$. Then there exist $C=C(\zeta_2)>0$ and $c=c(\zeta_2)>0$ such that for any $0<\varepsilon <T/2$, $\mathfrak{g}\in \mathcal{G}$ and $u(\mathfrak{g})\in H^{2,1}(Q)\cap H^1((0,T);H^1(B))$ satisfying \eqref{inIBVP} we have
\[
\|\mathfrak{g}\|_{H^1((\varepsilon,T-\varepsilon)\times \mathcal{S})}
\le 
Ce^{c/\varepsilon}\left(\|u(\mathfrak{g})\|_{H^1(\Sigma_0)}+\|\partial_\nu u(\mathfrak{g})\|_{L^2(\Sigma_0)}\right).
\]
\end{thm}

Before proving Theorem \ref{inHolder}, we consider the particular case where  $\mathfrak{g}(t,x)=\mathfrak{g}_0(t)\mathfrak{g}_1(x)$, where $\mathfrak{g}_0$ is known. For this case, we have the following result.

\begin{cor}\label{corPa}
For any $\mathfrak{g}=\mathfrak{g}_0\otimes\mathfrak{g}_1\in \mathcal{G}$ we have
\[
\|\mathfrak{g}_1\|_{H^1(\mathcal{S})}
\le 
C\left(\|u(\mathfrak{g})\|_{H^1(\Sigma_0)}+\|\partial_\nu u(\mathfrak{g})\|_{L^2(\Sigma_0)}\right),
\]
where $C=C(g,T,B,\Omega,\beta/\alpha, \mathfrak{g}_0)$ is a constant.
\end{cor}


\begin{proof}
We observe that since $|\mathfrak{g}_0|\|\mathfrak{g}_1\|_{L^2(\mathcal{S})}\ge \alpha$, we have $\eta:=\min |\mathfrak{g}_0|>0$. Applying Theorem \ref{inHolder} with $\varepsilon =T/4$, we get
\[
\eta\sqrt{T/2}\left(\|\mathfrak{g}_1\|_{L^2(\mathcal{S})}+\|\nabla_\tau\mathfrak{g}_1\|_{L^2(\mathcal{S})}\right)\le Ce^{4c/T}\left(\|u(\mathfrak{g})\|_{L^2(\Sigma_0)}+\|\partial_\nu u(\mathfrak{g})\|_{L^2(\Sigma_0)}\right),
\]
where $C$ and $c$ are as in the statement of Theorem \ref{inHolder}.
\end{proof}

As in the elliptic case, the proof of Theorem \ref{inHolder} relies on a Carleman type inequality.

Let $D=B\setminus \overline{\Omega}$. From Lemma \ref{inB}, there exists $\phi\in C^2(\overline{D})$ satisfying
\[
\begin{cases}
 \phi>0,\quad \mbox{in}\; D,
\\
\phi_{|\mathcal{S}}=0,
\\
\displaystyle \delta:=\min_{\overline{D}}|\nabla\phi|>0.
\end{cases}
\]
Set $\ell (t)=1/[t(T-t)]$ and 
\[
\varphi(t,x):=(e^{\gamma(\phi(x)+2m)}-e^{4\gamma m})\ell (t),\quad \xi(t,x):=e^{\gamma(\phi(x)+2m)}\ell(t),
\]
where $\displaystyle m:=\max_{\overline{D}}\phi$.

\begin{prop}\label{global_Carleman_estimate_degenerate}
Let $\zeta_3=(g,T,B,\Omega,\phi)$. There exist $\gamma_\ast=\gamma_\ast(\zeta_3)>0$, $s_\ast=s_\ast (\zeta_3)>0$ and $C=C(\zeta_3)>0$ such that for any $\gamma \ge \gamma_\ast$, $s\ge s_\ast$ and $u\in H^{2,1}((0,T)\times D)\cap H^1((0,T);H^1(D))$ we have
\begin{align*}
&C\left(\int_{(0,T)\times D}e^{2s\varphi}(\gamma\sigma(|\nabla u|^2+\sigma^2|u|^2))dxdt+\int_{\Sigma} e^{2s\varphi}\sigma(|\partial_{\nu_g}u|^2+\sigma^2|u|^2)dSdt\right)
\\
&\hskip 1cm\le \int_{(0,T)\times D} e^{2s\varphi}|(\partial_t-\Delta_g)u|^2dxdt
\\
&\hskip 3cm+\int_{\Sigma_0}e^{2s\varphi}\left(\sigma^{-1}|\partial_t u|^2+\sigma|\nabla u|^2+\sigma^3|u|^2\right)dSdt
\\
&\hskip 5cm+\int_{\Sigma} e^{2s\varphi}\left(\sigma^{-1}|\partial_t u|^2+\sigma |\nabla_\tau u|^2\right)dSdt,
\end{align*}
where $\sigma=s\gamma\xi$.
\end{prop}

We omit the proof of Proposition \ref{global_Carleman_estimate_degenerate} since it is quite similar to that of Proposition \ref{global_Carleman_estimate}.

\begin{proof}[Proof of Theorem \ref{inHolder}]
Let $\mathfrak{g}\in \mathcal{G}$ and $u=u(\mathfrak{g})$. Then we have
\begin{align}\label{Pa}
&\|\partial_t u(t,\cdot)\|_{L^2(\mathcal{S})}+\|\nabla_\tau u(t,\cdot)\|_{L^2(\mathcal{S})}
\\
&\hskip 1cm =\|\partial_t\mathfrak{g}(t,\cdot)\|_{L^2(\mathcal{S})}+\|\nabla_\tau \mathfrak{g}(t,\cdot)\|_{L^2(\mathcal{S})}\le (\beta/\alpha)\|\mathfrak{g}(t,\cdot)\|_{L^2(\mathcal{S})}.\notag
\end{align}

Henceforth, $C=C(\zeta_2)>0$ and $c=c(\zeta_2)>0$ denote generic constants. Fix $\gamma>\gamma_\ast$, where $\gamma_\ast$ is given by Proposition \ref{global_Carleman_estimate_degenerate}. Using \eqref{Pa}, $\varphi_{|\Sigma}=(e^{2\gamma m}-e^{4\gamma m})\ell$, $\xi_{|\Sigma}=e^{2\gamma m}\ell$ and applying Proposition \ref{global_Carleman_estimate_degenerate}, we get
\begin{align*}
&C\int_\Sigma e^{2s\varphi}\omega^3|\mathfrak{g}|^2dSdt
\\
&\hskip.5cm \le \int_{\Sigma_0}e^{2s\varphi}\omega^3\left(|\partial_t u|^2+|\nabla u|^2+|u|^2\right)dSdt+(\beta/\alpha)^2\int_\Sigma e^{2s\varphi}\omega |\mathfrak{g}|^2dSdt,\quad s\ge s_\ast,
\end{align*}
where $s_\ast=s_\ast(\zeta_2)>0$ is a constant and $\omega:=s\xi$.

Upon modifying $s_\ast$, we  may and do assume that $C\omega^3-(\beta/\alpha)^2\omega>(C/2)\omega^3$. In this case we have
\[
C\int_\Sigma e^{2s\varphi}\omega^3|\mathfrak{g}|^2dSdt
\le \int_{\Sigma_0}e^{2s\varphi}\omega^3\left(|\partial_t u|^2+|\nabla u|^2+|u|^2\right)dSdt, \quad s\ge s_\ast.
\]
Let $\varepsilon\in (0,T/2)$ be arbitrarily fixed. Since
\[
\ell(t)^{-1}=t(T-t)\ge T\varepsilon/2,\quad  t\in(\varepsilon,T-\varepsilon),
\]
we obtain
\[
\varphi(t,x)\ge 2(e^{\gamma(\phi(x)+2m)}-e^{4\gamma m})/(T\varepsilon)\ge -c/\varepsilon,
\]
which implies that
\[
e^{2s\varphi}\ge e^{-cs/\varepsilon}\quad \mbox{in}\; (\varepsilon,T-\varepsilon)\times\mathcal{S}.
\]
Therefore, it follows that
\[
\int_\Sigma e^{2s\varphi}\omega^3|\mathfrak{g}|^2dSdt  \ge Cs^3e^{-cs/\varepsilon}\|\mathfrak{g}\|_{L^2((\varepsilon,T-\varepsilon)\times\mathcal{S})}^2.
\]
Moreover, since
\[
e^{2s\varphi}\omega ^3\le e^{9\gamma m}s^3\ell^3e^{-2cs\ell}\le Ce^{-cs\ell},
\]
we find
\[\int_{\Sigma_0}e^{2s\varphi}\omega^3\left(|\partial_t u|^2+|\nabla u|^2+|u|^2\right)dSdt
\le C\left(\|u\|_{H^1(\Sigma_0)}^2+\|\partial_\nu u\|_{L^2(\Sigma_0)}^2\right).
\]
Combining these estimates yields
\[
\|\mathfrak{g}\|_{L^2((\varepsilon,T-\varepsilon)\times\mathcal{S})}^2\le Ce^{cs/\varepsilon}\left(\|u\|_{H^1(\Sigma_0)}+\|\partial_\nu u\|_{L^2(\Sigma_0)}\right)^2.
\]
\end{proof}

\section{Comments for quantitative uniqueness of continuation}

\subsection{Elliptic case}
We consider the same problem setup as in section \ref{section1}. By exactly the same proof as Theorem \ref{exLipschitz}, we obtain the following theorem.

\begin{thm}\label{quantitative_uniqueness}
Let $B$ and $\Omega$ be the open sets satisfying the same assumptions as in Theorem \ref{exLipschitz} and set $\zeta_0=(g,p,B,\Omega,\beta/\alpha)$. Then there exists $C=C(\zeta_0)>0$ such that for any $u\in H^2(U)$ satisfying $Pu=0$ in $U$ and $u_{|\mathcal{S}}\in \mathcal{A}$ we have
\[
\|u\|_{H^1(\Omega\setminus\overline{B})}+\|u\|_{H^1(\mathcal{S})}+\|\partial_{\nu_g} u\|_{L^2(\mathcal{S})}\le C\left(\|u\|_{H^1(\Gamma)}+\|\partial_\nu u\|_{L^2(\Gamma)}
\right).
\]
\end{thm}

Theorem \ref{quantitative_uniqueness} is nothing but the global Lipschitz stability for unique continuation with Cauchy data $(u_{|\Gamma},\partial_\nu u_{|\Gamma})$ under the constraint $u_{|\mathcal{S}}\in \mathcal{A}$. It is well known that Cauchy problem for elliptic equations is an ill-posed problem and its instability was pointed out by the example due to Hadamard \cite{Hadamard1923}. Conditional stability has been studied for many years to overcome this situation. There have been many studies on conditional stability, but as far as global stability estimates for $H^2$-solutions are concerned, the only known results are logarithmic and double logarithmic type stability by Bourgeois \cite{Bourgeois2010} and Choulli \cite{Choulli2020}. Most previous results on conditional stability assume a priori boundedness of a norm of solutions, and few studies have been done to replace this constraint. To the best of the authors' knowledge, Theorem \ref{quantitative_uniqueness} is the first result that improves stability by setting the new constraint $u_{|\mathcal{S}}\in \mathcal{A}$.

\subsection{Parabolic case}
Similarly for the parabolic case presented in section \ref{section4}, the following Lipschitz stability inequality can be obtained in $(\varepsilon,T-\varepsilon)\times (B\setminus\overline{\Omega})$. See \cite[Theorem 3.5.1]{Vessella2008} and \cite[Theorem 5.1]{Yamamoto2009} for a similar setup but for H\"{o}lder type stability estimate. We remark that in the result of \cite{Vessella2008}, the time interval is not truncated near the final time $T$. 

\begin{thm}\label{parabolic_quantitative_uniqueness}
Let $\zeta_2=(g,T,B,\Omega,\beta/\alpha)$. Then there exist $C=C(\zeta_2)>0$ and $c=c(\zeta_2)>0$ such that for any $0<\varepsilon <T/2$ and $u\in H^{2,1}(Q)\cap H^1((0,T);H^1(B))$ satisfying $(\partial_t-\Delta_g)u=0$ in $Q$ and $u_{|\Sigma}\in \mathcal{G}$ we have
\begin{align*}
&\|u\|_{L^2((\varepsilon,T-\varepsilon);H^1(B\setminus\overline{\Omega}))}+\|u\|_{H^1((\varepsilon,T-\varepsilon)\times \mathcal{S})}+\|\partial_{\nu_g}u\|_{L^2((\varepsilon,T-\varepsilon)\times\mathcal{S})}
\\
&\hskip 2cm \le  Ce^{c/\varepsilon}\left(\|u\|_{H^1(\Sigma_0)}+\|\partial_\nu u\|_{L^2(\Sigma_0)}\right).
\end{align*}
\end{thm}


\section*{Acknowledgement}
This work was supported by JSPS KAKENHI Grant Number JP22K20340 and JP23KK0049.

\section*{Declarations}
\subsection*{Conflict of interest}
The authors declare that they have no conflict of interest.

\subsection*{Data availability}
Data sharing not applicable to this article as no datasets were generated or analyzed during the current study.

\bibliographystyle{plain}
\bibliography{stability_continuation_17}

\end{document}